\documentclass[11pt,reqno]{amsart}

\usepackage[margin=0.9in]{geometry}
\usepackage{booktabs}
\usepackage[table]{xcolor}
\usepackage{tikz}
\usepackage{tkz-berge}

\usepackage{color}
\usepackage[colorlinks=true, linkcolor=blue, anchorcolor=blue, citecolor=blue, filecolor=blue, menucolor= blue, urlcolor=blue]{hyperref}
\usepackage{mathtools}

\definecolor{patriarch}{rgb}{0.5, 0.0, 0.5}
\newtheorem{theorem}{Theorem}[section]

\newtheorem{conj}[theorem]{Conjecture}

\newtheorem{lemma}[theorem]{Lemma}

\theoremstyle{definition}
\newtheorem{definition}[theorem]{Definition}

\newtheorem{observation}[theorem]{Observation}

\theoremstyle{remark}
\newtheorem{case}{Case}[theorem]

\newcommand{\cC}{\mathcal{C}}
\newcommand{\cG}{\mathcal{G}}

\newcommand{\w}{\omega}
\renewcommand{\a}{\alpha}
\newcommand{\D}{\Delta}

\newcommand{\ftdw}[3]{f_{#1}(#2,#3)}
\newcommand{\ftdwdef}{\ftdw{t}{\D}{\w}}
\newcommand{\cgdw}{\cG(\D,\w)}
\newcommand{\dx}{d_{\times}}

\newcommand{\N}{\mathbb{N}} 
\newcommand{\Z}{\mathbb{Z}}
\newcommand{\of}{\subseteq}
\newcommand{\wo}{\setminus}
\newcommand{\compl}{\overline}
\newcommand{\0}{\emptyset}

\DeclarePairedDelimiterXPP{\T}[2]{T_{#2}}{(}{)}{}{#1}
\DeclareMathOperator{\BT}{BT}
\newcommand{\BTt}{\widetilde{\BT}}

\DeclarePairedDelimiter\parens{(}{)}
\DeclarePairedDelimiter\floor{\lfloor}{\rfloor}
\DeclarePairedDelimiter\set{\{}{\}}
\DeclarePairedDelimiterX\setof[2]{\{}{\}}{#1\,:\,#2}
\DeclarePairedDelimiter\abs{|}{|}

\newcommand{\iso}{\simeq}

\title[Maximizing $K_t$ density]{Maximizing the density of $K_t$'s in graphs of bounded degree
and clique number}
\author[Kirsch and Radcliffe]{R. Kirsch and A.J. Radcliffe}

\begin{document}

\begin{abstract}
		Zykov showed in 1949 that among graphs on $n$ vertices with clique number $\w(G) \le \w$, the Tur\'an graph $\T{n}{\w}$ maximizes not only the number of edges but also the number of copies of $K_t$ for each size $t$. The problem of maximizing the number of copies of $K_t$ has also been studied within other classes of graphs, such as those on $n$ vertices with maximum degree $\D(G) \le \D$.

We combine these restrictions and investigate which graphs with $\D(G) \le \D$ and $\w(G) \le \w$ maximize the number of copies of $K_t$ per vertex. We define $\ftdwdef$ as the supremum of $\rho_t$, the number of copies of $K_t$ per vertex, among such graphs, and show for fixed $t$ and $\w$ that $\ftdwdef = (1+o(1))\rho_t(\T{\D+\floor[\big]{\frac{\D}{\w-1}}}{\w})$. For two infinite families of pairs $(\D,\w)$, we determine $\ftdwdef$ exactly for all $t\ge 3$. For another we determine $\ftdwdef$ exactly for the two largest possible clique sizes. Finally, we demonstrate that not every pair $(\D,\w)$ has an extremal graph that simultaneously maximizes the number of copies of $K_t$ per vertex for every size $t$.
\end{abstract}

\maketitle
\section{Introduction}

Tur\'{a}n graphs maximize the number of edges among $n$-vertex graphs with bounded clique number $\w(G) \le \w$. Zykov generalized Tur\'{a}n's theorem, showing that the same graphs also maximize the number of copies of $K_t$ for each size $t$. We write $k_t(G)$ for this number.

\begin{theorem}[Zykov \cite{Z49}]\label{thm:Zykov}
	If $t\ge 2$, $G$ is a graph with $n$ vertices, and $\w(G)\leq \w$, then 
	\[
		k_t(G)\leq k_t(\T{n}{\w}),
	\]
	where $\T{n}{\w}$ is the Tur\'an graph with $\w$ parts. The extremal graph is unique whenever $t\le \min(\w,n)$. In particular, if $G$ has the maximum number of $K_t$'s for \emph{any} $t$ with $2\le t\le \min(\w,n)$ then it has the maximum number for \emph{all} such $t$. 
\end{theorem}

The problem of maximizing the number of cliques among $n$-vertex graphs with bounded maximum degree $\D(G) \le \D$ has also been studied. Cutler and Radcliffe \cite{CR13} showed that $aK_{\D+1}\cup K_b$, where $n = a(\D+1)+b$ and $0 \le b \le \D$, maximizes the total number of cliques, answering a question of Galvin \cite{G11}. Gan, Loh, and Sudakov \cite{GLS15} conjectured that the same graphs also maximize the number of copies of $K_t$ for each fixed size $t\ge 3$. They proved this conjecture for $a=1$, answering a question of Engbers and Galvin \cite{EG14}, and Cutler and Radcliffe proved the conjecture for $\D \le 6$. The analogue of this problem that fixes the number of edges instead of the number of vertices was investigated by the present authors in \cite{KR19}.

In this paper we combine the constraints $\w(G)\le \w$ and $\D(G) \le \D$ and define $\ftdwdef$ to be the maximum number of $K_t$'s per vertex among graphs in this class. To be precise, we let 
\[
	\ftdwdef = \lim_{n\to\infty}\frac{k_t(n,\D,\w)}{n},
\]
where $k_t(n,\D,\w)$ is the maximum number of $K_t$'s in a graph on $n$ vertices with maximum degree at most $\D$ and clique number at most $\w$. We prove carefully in Section \ref{prelim} that this limit exists, and also discuss our terminology and notation. In Section \ref{sec:bounds}, we give upper and lower bounds on $\ftdwdef$ for all triples $(t,\D,\w)$. The bounds are asymptotically equivalent as $\D\to \infty$. When $\w-1$ divides $\D$ the bounds agree and are achieved by certain Tur\'an graphs. 

In Section \ref{sec:d=w}, we find graphs that achieve $\ftdwdef$ exactly, when $\D=\w$. In Section \ref{sec:max}, we find graphs that achieve $\ftdwdef$ exactly, when $\D=\w+1$ and $t = \w-1$ or $\w$. Finally, we find $\ftdwdef$ for three other triples $(t,\D,\w)$ in Section \ref{sec:other}. These examples demonstrate that for fixed $\D$ and $\w$, it can be the case that different graphs maximize the number of $K_t$'s per vertex for different values of $t$. Several open problems remain and are discussed in Section \ref{sec:open}.

\section{Definitions and preliminaries}\label{prelim}

Most of our graph theoretic notation is standard. For instance, we write $\T{n}{r}$ for the $r$-partite Tur\'{a}n graph on $n$ vertices. We write $E_n$ for the empty graph on $n$ vertices and $G\vee H$ for the join of graphs $G$ and $H$. 

\begin{definition}
	The class of graphs we are concerned with is $\cG(\D,\w)$, those graphs $G$ with $\D(G) \le \D$ and $\w(G) \le \w$.
\end{definition}

\begin{definition}We write $k_t(G)$ for the number of complete subgraphs of size $t$ in $G$, and define
\[
	\rho_t(G) = \frac{k_t(G)}{|V(G)|} .
\]
We also consider local versions of the clique count. We define the \emph{$t$-weight} of a vertex $v$ to be the number of $K_t$'s containing $v$, denoted $k_t(v)$. The \emph{$t$-weight} of an edge $e$ is the number of $K_t$'s containing $e$, denoted $k_t(e)$.
\end{definition}

Let us define
\[
	k_t(n,\D,\w) = \max\setof{k_t(G)}{|V(G)| = n, G\in\cG(\D,\w)}.
\] 
It is clear that $k_t(n,\D,\w)$ is superadditive (i.e., that $k_t(x+y,\D,\w) \ge k_t(x,\D,\w) + k_t(y,\D,\w)$), since the disjoint union of the graphs $G$ and $H$ that achieve $k_t(x,\D,\w)$ and $k_t(y,\D,\w)$ respectively, is a graph on $x+y$ vertices with $G\cup H \in \cG(\D,\w)$ that has $k_t(G\cup H) = k_t(x,\D,\w) + k_t(y,\D,\w)$.
By Fekete's Lemma \cite{Fekete} we have 
\[
	\lim_{n\to\infty}\frac{k_t(n,\D,\w)}{n} = \sup\setof[\Big]{\frac{k_t(n,\D,\w)}{n}}{n \in \mathbb N} .
\] 
We denote the limit by $\ftdwdef$. This parameter is the topic of the paper. Our approach will often be to consider the $t$-weights of vertices $v$ and bound the average $t$-weight. The following lemma gives a best possible bound on the $t$-weight of a vertex in a graph from $\cgdw$. In the lemma and in the rest of the paper we write $N(v)$ and $N[v]$ for the graphs induced by the open and closed neighborhoods of $v$ respectively.

\begin{lemma}\label{lem:perfectvx}For $3\le t\le \min(\D+1,\w)$, a vertex $v$ in a graph $G \in \cG(\D,\w)$ has $k_t(v) \le k_{t-1}(\T{\D}{\w-1})$, with equality if and only if $N(v) \iso \T{\D}{\w-1}$.\end{lemma}

\begin{proof}
	The $K_t$'s containing $v$ correspond bijectively to the $K_{t-1}$'s in $N(v)$, so $k_t(v) = k_{t-1}(N(v))$. In particular, the size of the largest clique $\w(N(v)) \le \w-1$ because $\w(G) \le \w$. The neighborhood $N(v)$ has $d(v) \le \D$ vertices. Applying Zykov's theorem to $N(v)$ with $\D$ vertices and clique number at most $\w-1$ yields the result.
\end{proof}

This lemma makes possible the following definition.

\begin{definition}
	A vertex $v$ is \emph{$(\D,\w)$-optimal} if $N(v) \iso\T{\D}{\w-1}$. If $\D$ and $\w$ are clear from context, we will just say \emph{optimal}. By Lemma \ref{lem:perfectvx}, $v$ is optimal if it has the maximum weight $k_t(v) = k_{t-1}(\T{\D}{\w-1})$ for any (and hence all) clique size(s) $3\le t \le \min(\D,\w-1)$.
\end{definition}

\begin{observation}\label{obs:induction} We will often consider an induced subgraph $H$ of $G$, for instance that induced by the closed neighborhood of an optimal vertex, and think about how it is connected to the rest of the graph. If we are in luck, for every $e \in E(H, G\setminus H)$, $k_t(e) = 0$. We call such a subgraph \emph{detachable} since
\[
	k_t(G) = k_t(H) + k_t(G\setminus H).
\] 
Now if $\rho_t(G\wo H)\le \a$ by induction and $\rho_t(H) \le \a$, then 
\[
    \rho_t(G) = \frac{k_t(G)}{n} = \frac{k_t(H) + k_t(G\wo H)}{n} \le \frac{\a\abs{H} + \a(n-\abs{H})}{n} = \a.
\]
\end{observation}

The following definition describes the places in $H$ where there might be edges to $G\wo H$.

\begin{definition}
	Given a graph $G \in \cG(\D,\w)$ and an induced subgraph $H$ of $G$, we say that $v\in V(H)$ is a \emph{border vertex of $H$} if $d_H(v)<\D$. If $v$ is a border vertex we write $\dx(v)$ for the number of \emph{cross-edges} from $v$ to $G\wo H$; i.e. $\abs{N(v) \wo H}$.
\end{definition}

Bounds on the number of cross-edges at border vertices allow us to show a subgraph $H$ is detachable.

\begin{lemma}\label{lem:detach} For an induced subgraph $H$ of $G$, let $B$ be the subgraph of $G$ induced by the border vertices of $H$. Let $i = \w(B)$. Suppose each border vertex $v$ has $\dx(v)\le j$. Then for $t > i+j$, $H$ is detachable. If, in addition, each $i$-clique of $B$ has some vertex $v$ with $\dx(v)<j$, then for $t > i+j-1$, $H$ is detachable.
\end{lemma}

\begin{proof} Suppose $v$ is a border vertex of $H$. Any $K_t$ containing both $v$ and a vertex of $G\wo H$ must meet $H$ only in vertices of $B$, since the non-border vertices of $H$ cannot have neighbors outside $H$. There are at most $i+j$ possible vertices in any $K_t$ that contains both $v$ and a vertex outside $H$: at most $i$ vertices in a largest clique in $B$, and at most $j$ vertices outside $H$ that are adjacent to $v$. 
	
	If, in addition, each $i$-clique of $B$ has some vertex with at most $j-1$ cross-edges, then the largest size clique that can contain a cross-edge is at most $i+j-1$: either at most $i$ vertices in $B$ with at most $j-1$ vertices outside $H$, or at most $i-1$ vertices in $B$ with at most $j$ vertices outside $H$.
\end{proof}

Another technique we will use is to translate bounds on all the $k_t(v)$'s in a graph $G$ into bounds on $k_t(G)$. The following lemma makes this explicit.

\begin{lemma}\label{lem:ktvbound}
	For any graph $G$ if $k_t(v) \le m$ for all $v\in V(G)$ then $\rho_t(G) \le m/t$.
\end{lemma}

\begin{proof}
	By counting in two ways the pairs $(v, K)$ consisting of a vertex $v$ and a subset $K\of V(G)$ containing $v$ inducing a complete graph of size $t$, we obtain 
	\[
		tk_t(G) = \sum_{v \in V(G)} k_t(v) \le m\abs{V(G)}. \qedhere
	\]
\end{proof}

\section{Bounds on $\ftdwdef$ and Tur\'an Graphs}\label{sec:bounds}

For every $t$, $\D$, and $\w$, there is a Tur\'an graph $G$ with $\D(G) \le \D$ and $\w(G) \le \w$ that gives a lower bound on $\ftdwdef$. A different Tur\'an graph bounds the number of $k_{t-1}$'s at each vertex, giving an upper bound on $\ftdwdef$. We use these bounds to determine the asymptotic behavior of $\ftdwdef$ for fixed $t$ and $\w$ as $\D\to \infty$. When $\w-1$ divides $\D$, we determine $\ftdwdef$ exactly and give the extremal graphs.

We will use the following count of $K_t$'s in Tur\'an graphs for both bounds.

\begin{lemma}\label{lem:count}The number of copies of $K_t$ in the Tur\'an graph $\T{n}{r}$ is 
	\[
		k_t(\T{n}{r}) = \sum_{k=0}^c\binom{c}{k}\binom{r-c}{t-k} (q+1)^k q^{t-k},
	\]
where $n = qr+c$ and $0 \le c < r$.
\end{lemma}

\begin{proof}
The graph $\T{n}{r}$ has $r$ parts, $c$ of size $q+1$ and $r-c$ of size $q$. The terms in the sum count the number of $K_t$'s having $k$ vertices in parts of size $q+1$.
\end{proof}

\begin{definition}
Given $\D$ and $\w$, we define the \emph{lower bound graph} $L(\D,\w)$ as follows. First, we define $a$ and $b$ by 
\begin{equation}\label{eq:ab}
	\D = a(\w-1)+b \qquad \text{and} \qquad 0 \le b < \w-1.
\end{equation}

We let $L(\D,\w)$ be $\T{\D+a}\w$, which we will show belongs to $\cgdw$. Throughout this section we will use $a$ and $b$ as defined in (\ref{eq:ab}).
\end{definition}

It is often, but not always, the case that the lower bound graph attains $\ftdwdef$. Of course, it always serves to give us a lower bound.

\begin{lemma}\label{lem:lower} For all $t,\D, \w\ge 2$ we have
\[
    \ftdwdef \ge \rho_t\parens{\T{\D+a}{\w}} 
			  = \frac{1}{a\w+b}\sum_{k=0}^b\binom{b}{k}\binom{\w-b}{t-k}(a+1)^ka^{t-k}. 
\]
\end{lemma}

\begin{proof} Set $n=a\w+b = \D+a$ and let $G=\T{n}{\w}$. Clearly $\w(G) \le \w$. The maximum degree of $G$ is achieved by any vertex in a part of size $a$, and is
\[
    \D(G) = n-a = \D.
\]
Therefore $G \in \cG(\D,\w)$ and $\ftdwdef \ge \rho_t(G)$. Lemma \ref{lem:count} with $r = \w$ gives 
\[
    k_t(G) = \sum_{k=0}^b\binom{b}{k}\binom{\w-b}{t-k}(a+1)^k(a)^{t-k}. \qedhere
\]

\end{proof}

For the upper bound on $\ftdwdef$, we will use the upper bound on the number of $K_t$'s at a single vertex.

\begin{lemma}\label{lem:upper} For all $t,\D, \w\ge 2$ we have
\[
    \ftdwdef \le \frac{1}{t}\, k_{t-1}(\T{\D}{\w-1}) =                     
                    \frac{1}{t}\sum_{k=0}^b\binom{b}{k}\binom{\w-1-b}{t-1-k}(a+1)^ka^{t-1-k}.
\]
\end{lemma}

\begin{proof}
    By Lemma \ref{lem:perfectvx}, if $G\in\cgdw$ then for every vertex $v$ we have 
\[
    k_t(v)\le k_{t-1}(\T{\D}{\w-1}).
\]
Applying Lemma \ref{lem:ktvbound} and Lemma \ref{lem:count} we get
\[
    \rho_t(G) \le \frac{1}{t}\, k_{t-1}(\T{\D}{\w-1}) 
			  = \frac{1}{t} \sum_{k=0}^b\binom{b}{k}\binom{\w-1-b}{t-1-k}(a+1)^k(a)^{t-1-k}. \qedhere
\]
\end{proof}

As an immediate consequence, we have the following exact bounds.

\begin{theorem}\label{thm:div} 
	For all $2\le t\le \w$, if  $\w-1$ divides $\D$, then 
	\[
	    \ftdwdef = \rho_t\parens{L(\D,\w)}.
	\]
\end{theorem}

\begin{proof}
	When $\w-1$ divides $\D$, we have $b=0$ and $a=\D/(\w-1)$, so by Lemmas \ref{lem:lower} and \ref{lem:upper} we get 
	\[
		\frac{1}{a\w}\binom{\w}{t}a^t = \rho_t\parens{\T{\D+a}{\w}}  \le \ftdwdef \le \frac{1}{t}\, k_{t-1}(\T{\D}{\w-1}) = \frac{1}{t}\binom{\w-1}{t-1}a^{t-1}. 
	\]
    Since the two ends of this string of inequalities agree we have equality throughout.
\end{proof}

As $\D\to\infty$, the upper and lower bounds from the previous two lemmas agree asymptotically.

\begin{theorem}\label{thm:asym}
	For fixed $t$ and $\w$, and $\D\to\infty$, 
	\[
		\ftdwdef = (1+o_\D(1))\rho_t\parens[\Big]{\T[\big]{\D+a}{\w}} 
				 = (1+o_\D(1))\, \frac1{t}  \binom{\w-1}{t-1} \parens[\Big]{\frac{\D}{\w-1}}^{t-1}.
	\]
\end{theorem}

\begin{proof}
	Lemmas \ref{lem:lower} and \ref{lem:upper} together give lower and upper bounds on $\ftdwdef$. Let $a,b$ be defined as in equation (\ref{eq:ab}) as functions of $\D$ and $\w$. As $\D\to\infty$, our upper bound (divided by $\D^{t-1}$) tends to
\begin{align*}
    \lim_{\D\to\infty} \frac{1}{\D^{t-1}}\frac{1}{t} \,k_{t-1}(\T{\D}{\w-1}) 
        &= \lim_{\D\to\infty} \frac{1}{\D^{t-1}} \frac{1}{t} \sum_{k=0}^b \binom{b}{k} \binom{\w-1-b}{t-1-k} 
                                                                        (a+1)^ka^{t-1-k}\\
        &= \lim_{\D\to\infty} \frac{1}{\D^{t-1}} \frac{1}{t} \sum_{k=0}^b \binom{b}{k} \binom{\w-1-b}{t-1-k}            
                                                                        \parens[\bigg]{\frac{\D}{\w-1}}^{t-1}\\
        &= \parens[\bigg]{\frac{1}{\w-1}}^{t-1}\; \frac{1}{t} \binom{\w-1}{t-1}.
    \end{align*}
On the other hand the limiting ratio of the lower bound is
\begin{align*}
    \lim_{\D\to\infty} \frac{1}{\D^{t-1}} \rho_t\parens[\Big]{\T[\big]{\D+a}{\w}} 
        &= \lim_{\D\to\infty} \frac{1}{\D^{t-1}} \frac{1}{\D+a} 
                \sum_{k=0}^b \binom{b}{k} \binom{\w-b}{t-k} (a+1)^k  a^{t-k} \\
        &= \lim_{\D\to\infty} \frac{1}{\D^{t-1}} \frac{\w-1} {\D\w}
                \sum_{k=0}^b\binom{b}{k}\binom{\w-b}{t-k} \parens[\bigg]{\frac{\D}{\w-1}}^t \\
        &= \parens[\bigg]{\frac{1}{\w-1}}^{t-1} \; \frac{1}{\w} \binom{\w}{t} . 
\end{align*}

Since $\frac{1}{\w} \binom{\w}{t} = \frac{1}{t} \binom{\w-1}{t-1}$ the limits agree.
\end{proof}

Sufficiently small graphs $G$ cannot have greater $K_t$ density than the lower bound graph $L(\D,\w)$. To prove this we will use the following lemma that $\rho_t(\T{n}\w)$ weakly increases with $n$.

\begin{lemma}\label{lem:rhoincr}
	For given $t$, $\w$, and $n \le m$, $\rho_t(\T{n}\w) \le \rho_t(\T{m}\w)$.
\end{lemma}

\begin{proof}
	Let $n = q\w+c$ with $0 \le c \le \w-1$. Observe that for an $n$-vertex graph $G$, $\rho_t(G)$ increases with $n$ if and only if $t\rho_t(G) = \frac{1}{n}\sum_{v\in V(G)}k_t(v)$, the average $t$-weight, does. In $\T{n}\w$, there are $c(q+1)$ vertices with $t$-weight $k_t(v) = k_{t-1}(\T{n-q-1}{\w-1})$ and $(\w-c)q$ vertices with $t$-weight $k_t(v) = k_{t-1}(\T{n-q}{\w-1})$. Note that $k_{t-1}(\T{n-q-1}{\w-1}) \le k_{t-1}(\T{n-q}{\w-1})$. Let 
	\[
		\alpha = (c(q+1)k_{t-1}(\T{n-q-1}{\w-1}) + (\w-c)qk_{t-1}(\T{n-q}{\w-1}))/n = t\rho_t(\T{n}\w);
	\]since $\alpha$ is the average $t$-weight, we have $\alpha \le k_{t-1}(\T{n-q}{\w-1})$. Let $v$ be a vertex in $\T{n+1}\w$ such that $\T{n+1}\w - v \iso \T{n}\w$. Then $k_t(v) = k_{t-1}(\T{(n+1)-(q+1)}{\w-1}) \ge \alpha$. The other vertices of $\T{n+1}\w$ have at least as many $K_t$'s in $\T{n+1}\w$ as in $\T{n+1}\w - v \iso \T{n}\w$. Therefore $t\rho_t(\T{n}\w) \le t\rho_t(\T{n+1}\w)$, and $\rho_t(\T{n}\w) \le \rho_t(\T{m}\w)$.
\end{proof}

\begin{lemma}\label{lem:basecase}
	For an $n$-vertex graph $G \in \cgdw$ with $n \le \D + a$ and $t \ge 2$, $\rho_t(G) \le \rho_t(L(\D,\w))$.
\end{lemma}

\begin{proof}
	By Zykov's Theorem, $k_t(G) \le k_t(\T{n}{\w})$, and dividing by $n$, $\rho_t(G) \le \rho_t(\T{n}{\w})$. By Lemma \ref{lem:rhoincr}, $\rho(\T{n}\w) \le \rho_t(\T{\D+a}\w) = \rho_t(L(\D,\w))$.
\end{proof}

\section{Computing $\ftdw{t}{r}{r}$}\label{sec:d=w} 

In this section we show that for all $r\in \N$ and $t\ge 3$ the optimal ratio $\ftdw{t}{r}{r}$ is achieved by $\T{r+1}{r} \iso K_{r+1}-e$. We start with several lemmas addressing the case $t=3$. In outline, we'll show that vertices in a graph $G\in \cG(r,r)$ are either optimal, seriously suboptimal, or are contained in a local structure in which the average value of $k_3(v)$ is low. We discuss this borderline case first, which requires a definition, local to Section \ref{sec:d=w}. Later, in Section \ref{sec:other}, we reuse the word configuration to refer to a similar local situation.

\begin{definition}
    If $G\in \cG(r,r)$, $C \of V(G)$ has size $r+1$, and $G[C]$ is complete except for (exactly) two missing edges then we call $C$ a \emph{full configuration} in $G$.
\end{definition}

\begin{lemma}\label{lem:config}
	If $r\ge 6$, $G\in \cG(r,r)$, and $C$ is a full configuration in $G$, then the average 3-weight of vertices in $C$ is at most $\binom{r}2-3$, i.e.,
	\[
		\sum_{v\in C} k_3(v) \le (r+1)\parens[\Big]{\binom{r}2 - 3}.
	\]
	[Here we are computing $k_3(v)$ in $G$, not in $G[C]$.]
\end{lemma}

\begin{proof}
Let $e$ and $f$ be the two edges of $\compl{G[C]}$. 
\begin{case} $e$ and $f$ are incident.

	Let $w$ be the vertex common to $e$ and $f$, and let $x$ and $y$ be the other vertices in $e\cup f$. By the degree condition $w$ has at most two neighbors outside $C$, and each of $x$ and $y$ has at most one neighbor outside $C$. We see that 
	\begin{align*}
	    k_3(w) &\le \binom{r-2}2 + 1 = \binom{r}2 -  2(r-2),\ \text{and} \\
		k_3(x), k_3(y) &\le \binom{r-1}2 + 1 = \binom{r}2 - (r-2).
	\end{align*}
	For the first, note that $w$ is in a triangle with every pair from $C\wo \set{w,x,y}$ and at most one with its external neighbors. The vertex $x$ (and similarly $y$) is in a triangle with each pair from $C\wo \set{x,w}$ and at most one with $y$ and a common exterior neighbor. All $r-2$ remaining vertices in $C$ are in exactly $\binom{r}2 - 2$ triangles. Thus
	\[
		\sum_{v\in C} k_3(v) \le (r+1)\binom{r}2 - 6(r-2) = (r+1)\parens[\Big]{\binom{r}2 - 3} +(15-3r) 
							 \le (r+1)\parens[\Big]{\binom{r}2 - 3}, 
	\]
	for $r\ge 5$.
\end{case}

\begin{case}$e$ and $f$ are not incident.

	Let $e = \set{x,y}$. Then $x$ has at most one neighbor outside $N[v]$ and
	\[
		k_3(x) \le \parens[\Big]{\binom{r-1}2-1} + 2 = \binom{r}2 - (r-2)
	\]
	since $x$ is in a triangle with every pair from $C\wo e$ except for the pair $f$ and also at most two triangles involving its exterior neighbor. The same calculation applies to every vertex in $e\cup f$, and each of the other $r-3$ vertices of $C$ is in exactly $\binom{r}2-2$ triangles. Thus
	\[
		\sum_{v\in C} k_3(v) \le (r+1)\binom{r}2 - 4(r-2) - 2(r-3) = (r+1)\parens[\Big]{\binom{r}2 - 3} +(17-3r) 
							 \le (r+1)\parens[\Big]{\binom{r}2 - 3}, 
	\]
	for $r\ge 6$.
\end{case}
\end{proof}

The following lemma establishes the useful fact that full configurations are either identical or disjoint.

\begin{lemma}\label{lem:disjoint}
    If $r\ge 6$ and $G\in \cG(r,r)$ then any two distinct full configurations in $G$ are disjoint.
\end{lemma}

\begin{proof}
    Suppose that full configurations $C$ and $D$ have $C \cap D \neq \0$; we will show that $C=D$. Let $R = \compl{G[C\cup D]}$, $I = \compl{G[C\cap D]}$, and $i = |C\cap D| \ge 1$. Let $v \in C\cap D$ be a vertex with the minimum value of $d_R(v)$ among vertices in $C\cap D$. Let $\delta_I$ be this minimum value. Then 
    \[
    	r \ge d_G(v) \ge \abs{C\cup D} - 1 - \delta_I = (2r+2-i) - 1 - \delta_I,
    \] 
    so $i \ge r+1-\delta_I$. In particular if $\delta_I = 0$ then $C=D$.
    
    We will prove that $\delta_I \le 2$ by first showing that $\delta_I \ge 3$ implies $i = 1$. In $R$ a vertex $v \in I$ has at most two neighbors in $C$ and at most two neighbors in $D$ because they are full configurations. If $d_R(v) \ge 3$, then in $R$, $v$ has a neighbor not in $C$ and a neighbor not in $D$. If $i \ge 2$ and $\delta_I \ge 3$ then $R$ contains $2P_3$ as a subgraph with the two central vertices in $I$ and two endpoints in each of $C\setminus D$ and $D\setminus C$. This $2P_3$ in $R$ has four edges, but $R$ has at most four edges, so $R \iso 2P_3$, contradicting $\delta_I \ge 3$. Thus if $\delta_I \ge 3$, then $1 = i \ge r+1-4 = r-3$, contradicting $r \ge 6$.
	
	Since $R$ has at most $4$ edges, we claim that if $\delta_I \ge 1$, then $i\le 4$. Suppose $\delta_I \ge 1$, i.e. every vertex of $I$ is an endpoint of an edge in $R$. There are exactly two edges of $R$ in each of $C$ and $D$, so we have either (i) exactly two edges in $R$, both in $I$, so $i=3$ or $4$, (ii) exactly three edges in $R$, one in $I$, so $i=2$, $3$, or $4$, or (iii) exactly four edges in $R$, none in $I$, so $i=4$. In all cases $i \le 4$. In summary, $4 \ge i \ge r+1-\delta_I$. If $\delta_I$ is $1$ or $2$ then $4 \ge i \ge r-1$, contradicting $r \ge 6$. This leaves only the possibility that $\delta_I=0$.    
\end{proof}

Now we prove most of the cases when $t=3$.

\begin{lemma}\label{lem:t36}
	For $r \ge 6$, $\ftdw{3}{r}{r} = \rho_3(\T{r+1}{r})$.
\end{lemma}

\begin{proof}
First note that $\rho_3(G) \le \rho_3(\T{r+1}r)$ for all graphs $G$ on at most $r+1$ vertices by Lemma \ref{lem:basecase}. 
Suppose then that $G \in \cG(r,r)$ has at least $r+2$ vertices. Consider first the case that there is an optimal vertex $v$ in $G$. Then $N[v] \iso\T{r+1}{r}$, a $K_{r+1}$ minus an edge. There are at most two possible edges between $N[v]$ and $G\setminus N[v]$, at most one incident to each of the non-adjacent vertices, and none at the other vertices of $N[v]$, since $\D(G) \le r$. Thus $N[v]$ is detachable. By Observation \ref{obs:induction} we may apply induction, so $\rho_3(G) \le \rho_3(\T{r+1}{r})$.

Now consider the case where $G$ has no optimal vertex. Every vertex $v$ is either of degree at most $r-1$ or has at least two missing edges in its neighborhood.  Thus for all $v\in V(G)$ we have 
\[
	k_3(v) \le \max\parens[\Big]{\binom{r}{2}-2, \binom{r-1}2} = \binom{r}{2}-2.
\]
If $v$ achieves this upper bound on $k_3(v)$ then $v$ is in a full configuration---$N[v]$ to be precise.
Thus from Lemmas~\ref{lem:config} and \ref{lem:disjoint} we get
\begin{align*}
    \sum_{v\in V(G)} k_3(v) &= \sum_{\text{$v$ not in any full configuration}} k_3(v) + \sum_{\text{$C$ a full configuration}} \sum_{v\in C} k_3(v) \\
							 &\le n(G) \parens[\Big]{\binom{r}2 - 3},
\shortintertext{so}
	\rho_3(G) &\le \frac13\parens[\Big]{\binom{r}{2} -3} = \frac{1}{r+1}\binom{r+1}{3} - 1 \\
			  &< \frac1{r+1}\parens[\Big]{\binom{r+1}{3}-(r-1)}  = \rho_3(\T{r+1}{r}). \qedhere
\end{align*}
\end{proof}

All remaining cases are covered by a uniform argument that is the main focus of the proof of the following theorem.

\begin{theorem}\label{thm:rr}
	For all $3\le t\le r$ we have $\ftdw{t}{r}{r} = \rho_t(\T{r+1}{r})$.
\end{theorem}

\begin{proof}
	First note that $\rho_t(G) \le \rho_t(\T{r+1}r)$ for all graphs $G$ on at most $r+1$ vertices by Lemma \ref{lem:basecase}. Consider then $G \in \cG(r,r)$, a graph on $n$ vertices with $n>r+1$. Suppose first that there is an optimal vertex $v$ in $G$. Then $N[v] \iso\T{r+1}{r}$. As in the proof of Lemma~\ref{lem:t36} $N[v]$ is detachable, so by induction $\rho_t(G)\le \rho_t(\T{r+1}{r})$. 
    
Henceforth we may assume that there is no optimal vertex in $G$. Thus any vertex $v$ with $d(v)=r$ has at least two missing edges in its neighborhood. Thinking of $N(v)$ as $\T{r}{r-1}$ with at least one edge removed, and noting that every edge in $\T{r}{r-1}$ belongs to some copy of $K_{r-1}$ we have
\[
    k_t(v) = k_{t-1}(N(v)) \le k_{t-1}(\T{r}{r-1}) - \binom{r-3}{t-3}= \binom{r}{t-1} - \binom{r-2}{t-3} - \binom{r-3}{t-3}.
\]
On the other hand if $d(v)\le r-1$ then $k_t(v) \le \binom{r-1}{t-1}$. Observing that 
\[
    \binom{r}{t-1} - \binom{r-1}{t-1} = \binom{r-1}{t-2} = \binom{r-2}{t-3} + \binom{r-2}{t-2}
                                      \ge \binom{r-2}{t-3} + \binom{r-3}{t-3}, 
\]
we in fact have that 
\[
    k_t(v) \le \binom{r}{t-1} - \binom{r-2}{t-3} - \binom{r-3}{t-3}
\]
for all $v\in V(G)$. Now Lemma \ref{lem:ktvbound} implies that $\rho_t(G) \le \parens[\big]{\binom{r}{t-1} - \binom{r-2}{t-3} - \binom{r-3}{t-3}}/t$ and it only remains to show that
\[
    \frac{\binom{r}{t-1} - \left(\binom{r-2}{t-3} + \binom{r-3}{t-3}\right)}{t} 
        \le \frac{\binom{r+1}{t} - \binom{r-1}{t-2}}{r+1} = \rho_t(\T{r+1}{r}) .
\]
The leading terms on the two sides sides cancel, so actually we need to prove that 
\[
    t\binom{r-1}{t-2} \le (r+1)\parens[\Big]{\binom{r-2}{t-3} + \binom{r-3}{t-3}}.
\]
Clearing fractions (by multiplying by $(t-2)!$) and cancelling the factor $(r-3)(r-4)\cdots(r-t+2)$ common to all terms, we are reduced to showing that 
\begin{align*}
    t(r-1)(r-2) &\le (r+1)(t-2)\parens[\big]{(r-2) + (r-t+1)} = (r+1)(t-2)(2r-t-1), \\
\shortintertext{i.e., that}
              0 &\le -(r+1)t^2 + (r^2 + 6r - 1)t + (-4r^2-2r+2), \\
\shortintertext{i.e., that}
4 -\frac{6}{r+1}&\le t \le r+1.     
\end{align*}
This clearly holds for $t\ge 4$. It only holds when $t=3$ if $r\le 5$, but fortunately Lemma~\ref{lem:t36} covers the remaining cases.
\end{proof}

\section{Computing $\ftdw{r-1}{r+1}{r}$ and $\ftdw{r}{r+1}{r}$}\label{sec:max}

In this section we make progress on computing $\ftdw{t}{r+1}{r}$. We show that for the two largest possible sizes, $t = r-1$ and $t=r$, the extremal graphs are the lower bound graph $L(r+1,r)$.

The following lemma describes the possibilities for the closed neighborhood of an almost optimal vertex in the graphs we are considering. 

\begin{lemma}\label{lem:nhd}
    For $r \ge 3$, if $N$ is a graph with $|V(N)| \le r+2$, $\w(N) \le r$, and $k_r(N) = 3$, then $N = K_{r+2}-K_3$ or $N = K_{r+2}-P_4$.
\end{lemma}

The proof of this comes immediately after we prove the following translated version. 

\begin{lemma}\label{lem:comp}
	If $C$ is a graph with no vertex cover of size at most $1$ and exactly $3$ vertex covers of size $2$ then $C$ is either the union of $K_3$ with isolated vertices or the union of $P_4$ with isolated vertices.
\end{lemma}

\begin{proof}
    We first show that $C$ contains either a $K_3$ or two disjoint edges. If $\D(C)\le 1$ then $C$ must contain two disjoint edges else it has a vertex cover of size at most $1$. Otherwise pick $v$ with $d(v)\ge 2$. Since $\set{v}$ is not a vertex cover there must be an edge $e$ of $C$ not incident with $v$. Either $e$ joins two neighbors of $v$, and $C$ contains a $K_3$, or $e$ is disjoint from some edge incident with $v$. 

If $C$ contains a $K_3$ then we know that every vertex cover of $C$ of size $2$ must contain (and hence consist of) two vertices of the triangle. Since $C$ has three vertex covers of size two, we can pick any pair of vertices from the triangle. Now $C$ cannot contain another edge; the triangle edges are the only ones covered by all three of these pairs. Hence $G$ is the union of $K_3$ and isolates.

On the other hand, if $K_3 \not\of G$ and $2K_2\of C$, all vertex covers of size $2$ must contain (and hence consist of) one vertex from each edge of the $2K_2$. Since three of the possible pairs are vertex covers of $C$ there must be exactly one other edge of $C$ contained in the vertex set of the $2K_2$, making that induced graph a $P_3$. Again, the only potential pairs that are covered by all three vertex covers of size $2$ of the $P_4$ are exactly the $P_4$ edges, so $C$ must be $P_4$ with isolated vertices.
\end{proof}

\begin{proof}[Proof of Lemma~\ref{lem:nhd}]
	Since $k_r(N)=3$, $N$ cannot have $r$ or fewer vertices. If $N$ has $r+1$ vertices there must be three vertices $v_1,v_2,v_3$ such that each $N\wo\set{v_i}$ is complete, and hence $N$ itself is complete on $r+1=3$ vertices, contradicting $r\ge 3$. Thus $N$ has $r+2$ vertices. Now letting $C = \compl{N}$, we see that $C$ satisfies the conditions of Lemma~\ref{lem:comp} and we are done.
\end{proof}

\begin{theorem}\label{thm:rr+1r}For $r \ge 4$, $\ftdw{r}{r+1}{r} = \rho_r(\T{r+2}{r}) = \frac{4}{r+2}$.
\end{theorem}

\begin{proof} First note that $\rho_r(G) \le \rho_r(\T{r+2}{r})$ for all graphs $G$ on at most $r+2$ vertices by Lemma \ref{lem:basecase}. Let $G$ be a graph in $\cG(r+1,r)$. An optimal vertex $v$ in $G$ would have $N(v) \iso\T{r+1}{r-1}$. Since $r \ge 4$, we have $\frac{r+1}{r-1} < 2$, and so $\T{r+1}{r-1}$ has $r-3$ parts of size 1 and two parts of size 2. The closed neighborhood $N[v]$ would be isomorphic to $ \T{r+2}{r}$, which has $r-2$ parts of size 1 and two parts of size 2. Thus we would have $k_r(v) = 4$. By Lemma~\ref{lem:perfectvx}, only  optimal vertices achieve $k_r(v) = 4$.

\begin{case} $G$ has an optimal vertex $v$.\end{case}

We claim that the closed neighborhood $N[v] \iso\T{r+2}{r}$ is detachable. The border vertices are the four vertices in the parts of size 2; the largest clique on them is a $K_2$. A border vertex has at most 1 cross-edge. Since $2 + 1 < r$, Lemma \ref{lem:detach} shows $N[v]$ is detachable. By Observation~\ref{obs:induction} we are done by induction.

\begin{case}$G$ has a vertex $v$ such that $k_r(v) = 3$.\end{case}

It is easy to check that in this case $N=N[v]$ satisfies the hypotheses of Lemma~\ref{lem:nhd}, and therefore $N \iso K_{r+2} - K_3$ or $N \iso K_{r+2} - P_4$. We will show that in either case $N$ is detachable. If $N \iso K_{r+2} - K_3$, then the independent set of border vertices $B$ has clique number 1, and each border vertex has at most 2 cross-edges. Again, by Lemma \ref{lem:detach}, $N$ is detachable. 

If $N \iso K_{r+2} - P_4$, then border vertices $B$ of $N$ induce a $P_4$ (the complement of the one subtracted) with clique number $2$. Every edge of $B$ contains a vertex $w$ with $\dx(w)<2$. Therefore, by Lemma~\ref{lem:detach} $N$ is detachable for $t\ge 4$, and in particular for $r$. It is also suboptimal: $\rho_r(N) = \frac{3}{r+2} < \frac{4}{r+2} = \rho_r(\T{r+2}{r})$. Hence we are done by induction.

\begin{case} $k_r(v)\le 2$ for all vertices $v$ of $G$. \end{case} 
    
In this case we have (by Lemma~\ref{lem:ktvbound}) that 
\[
    \rho_r(G) \le \frac{2}{r} < \frac{4}{r+2} = \rho_r(\T{r+2}{r}). \qedhere
\]
\end{proof}

We are able to extend the methods of Lemma~\ref{lem:nhd} to deal with bounds on $k_{r-2}(N)$, and having done so we will extend the previous result to prove that $\T{r}{r+2}$ is also optimal for $\rho_{r-1}$ in $\cG(r+1,r)$. As before, the lemma is easier to understand in a complementary version.

\begin{lemma}\label{lem:nhd2} For $r \ge 5$, if $N$ is a graph on at most $r+1$ vertices such that $\w(N) \le r-1$ and
    \[
        4r-16+\frac{36}{r+2} \le k_{r-2}(N) < 4r-8,
    \] 
then $N \iso K_{r+1}-K_3$ or $N \iso K_{r+1}-P_4$.
\end{lemma}

\begin{lemma}\label{lem:comp2}
    Suppose $r\ge5$. If $C$ is a graph on $r+1$ vertices having no vertex cover of size $1$, such that the number of vertex covers of size $3$ in $C$, denoted $c_3$, satisfies 
    \[
        4r-16+\frac{36}{r+2} \le c_3 < 4r-8,
    \]
    then $C$ is either the union of $K_3$ with isolated vertices or the union of $P_4$ with isolated vertices.
\end{lemma}

\begin{proof}
        First note that, exactly as in the proof of Lemma~\ref{lem:comp}, $C$ contains either a $K_3$ or two disjoint edges. Now we'll show that $\D(C)\le 2$. If $v$ is a vertex of degree at least $3$ then there is at most one vertex cover of size $3$ not containing $v$ since it must contain each of $v$'s neighbors. On the other hand $C\wo \set{v}$ has some edge (since there is no dominating vertex) so the number of vertex covers of size $3$ containing $v$ is at most $2r-3$. To see this note that we have to augment $\set{v}$ with a vertex cover of size $2$ of $C\wo \set{v}$. Such covers must contain at least one vertex of the edge; there are $2r-3$ such sets. But for $r\ge 5$ we have
    \[
        2r-2 < 4r-16 + \frac{36}{r+2},
    \]
    contradicting our hypothesis on $c_3$. 
    
    Now, in similar fashion, we'll show that $C$ can't contain $3P_2$, $P_3\cup P_2$, or $C_4$. In the first case we'd have at most $8$ vertex covers of size $3$, and, for $r\ge 5$,
    \[
        c_3 \le 8 < 4r-16+\frac{36}{r+2}.
    \]
    For the next case, note that vertex covers of $C$ would have to contain two or three vertices from $P_3\cup P_2$, so, splitting the count according to which it is,
    \[
        c_3 \le 2(r-4) + 6 = 2r-2 < 4r-16+\frac{36}{r+2},
    \]
    for $r\ge 5$. Finally if $C_4\of G$ then
    \[
        c_3 \le 2(r-1) < 4r-16+\frac{36}{r+2},
    \]
    for $r\ge 5$. In any of the cases we get a contradiction.  
    
    Now we are almost done. If $K_3\of C$ then there can't be any other edges lest $C$ contain $P_3\cup P_2$ or a vertex of degree at least $3$. If $2P_2\of C$ then that can't be all the edges of $C$ for then we'd have $c_3 = 4(r-3)+4=4r-8 \not< 4r-8$. Any additional edge can only between the components of $2P_2$ else one of $P_3\cup P_2$ or $3P_2$ is a subgraph of $C$. Thus $P_4\of C$. The existence of any other edge would produce a forbidden subgraph. Thus the only non-trivial component of $C$ is a $K_3$ or a $P_4$.
\end{proof}

\begin{proof}[Proof of Lemma~\ref{lem:nhd2}]
    If $N$ has at most $r$ vertices then $N$ either is not complete or has fewer than $r$ vertices, so
    \[
        k_{r-2}(N) \le \max(k_{r-2}(K_r-e),k_{r-2}(K_{r-1})) = \max(2r-3,r-1) < 4r-16+\frac{36}{r+2},
    \]
    for $r\ge 3$, contradicting our hypotheses. On the other hand if $N$ has $r+1$ vertices then $C=\compl{N}$ satisfies the hypotheses of Lemma~\ref{lem:comp2}, and we are done.
\end{proof}

\begin{theorem}\label{thm:r-1r+1r}For $r \ge 5$, $\ftdw{r-1}{r+1}{r} = \rho_{r-1}(\T{r+2}{r})$.
\end{theorem}

\begin{proof} Let $G \in \cG(r+1,r)$ have $n$ vertices. First note that $\rho_{r-1}(G) \le \rho_{r-1}(\T{r+2}{r})$ for $n \le r+2$ by Lemma \ref{lem:basecase}.
	
	If $v$ is an optimal vertex of $G$ then we have $N(v) \iso\T{r+1}{r-1}$ which has (since $r\ge 5$) $r-3$ parts of size $1$ and two parts of size $2$. The closed neighborhood $N[v]$ is isomorphic to $\T{r+2}{r}$, which has $r-2$ parts of size 1 and two parts of size 2. Thus 
	\[
		k_{r-1}(v) = k_{r-2}(N(v)) = 4(r-3)+2\cdot 2 = 4r-8.
	\] 
	Moreover, by Lemma~\ref{lem:perfectvx} again, only optimal vertices have $k_{r-1}(v) = 4r-8$. Note that 
    \[
        \rho_{r-1}(\T{r+2}{r}) = \frac{4(r-2)+2\cdot 2}{r+2} = \frac{4r-4}{r+2}.
    \]
    We now split into cases depending on the weights of vertices in $G$.

\begin{case} $G$ has an optimal vertex $v$, i.e. one with $k_{r-1}(v) = 4r-8$.\end{case}

The largest clique on the border vertices of $N[v] \iso \T{r+2}{r}$ is a $K_2$. Each border vertex $w$ has $\dx(w)\le 1$. By Lemma \ref{lem:detach}, $N[v]$ is detachable for $t = r-1 \ge 4$, and (by Observation~\ref{obs:induction}) we are done by induction.

\begin{case} $G$ has a vertex $v$ such that 
	\[
		4r-16+\frac{36}{r+2} \le k_{r-1}(v) < 4r-8.
	\]
\end{case}

Note that $N=N(v)$ has at most $r+1$ vertices, contains no $r$-cliques, and, by hypothesis, satisfies 
\[
	4r-16+\frac{36}{r+2} \le k_{r-2}(N) < 4r-8.
\]
Since $r \ge 5$ we can apply Lemma \ref{lem:nhd2}, so $N \iso K_{r+1}-K_3$ or $N \iso K_{r+1}-P_4$. Consider first the case where $N \iso K_{r+1}-K_3$, so that $N[v] \iso K_{r+2}-K_3$. 
The border vertices of $N[v]$ in this case form an independent set. Each border vertex has at most two cross-edges. By Lemma \ref{lem:detach}, $N[v]$ is detachable for $t = r-1 \ge 4$. The number of $K_{r-1}$'s in $N[v]$ is $1 + 3(r-1)$: the first term counts the $K_{r-1}$'s not involving the border vertices, and the second counts the $K_{r-1}$'s with exactly one border vertex.

If, on the other hand $N[v] \iso K_{r+2}-P_4$, then the largest clique on the border vertices of $N[v]$ is of size $2$. The border vertices $w$ have $\dx(w)\le 2$, but at least one of each pair of adjacent border vertices has $\dx(w)\le 1$. By Lemma \ref{lem:detach}, $N$ is detachable for $t = r-1 \ge 4$. The number of $K_{r-1}$'s in $N[v]$ is $4 + 3(r-2)$: the first term counts the $K_{r-1}$'s with exactly one border vertex, and the second counts the $K_{r-1}$'s with exactly two border vertices.

For both possibilities, 
\[
	\rho_{r-1}(N[v]) = \frac{3r-2}{r+2} < \frac{4r-4}{r+2} = \rho_{r-1}(\T{r+2}{r}),
\]
for $r > 2$. We see that $N[v]$ is suboptimal and detachable, so $G$ is suboptimal by Obsertvation~\ref{obs:induction}.

\begin{case} Every vertex $v$ in $G$ has $k_{r-1}(v) < 4r-16+\frac{36}{r+2}$.\end{case}

By Lemma \ref{lem:ktvbound}, 
\[
	\rho_{r-1}(G) < \frac{4r-16+\frac{36}{r+2}}{r-1} = \frac{4r-4}{r+2} = \rho_{r-1}(\T{r+2}{r}),
\] 
so $G$ is suboptimal.
\end{proof}

\section{Other Values of $\D$ and $\w$}\label{sec:other}

We have seen that for pairs $(\D,\w)$ with $(\w-1)|\D$, in the case $\D=\w$, and for the triples $(t,\D,\w)$ studied in Section \ref{sec:max}, the lower bound graphs defined in Section \ref{sec:bounds} are extremal. In this section we give examples for which these lower bound graphs are not optimal. In particular we show that none of $\ftdw3{2k+1}3$, $\ftdw354$, and $\ftdw365$ are attained at $L(\D,\w)$. Moreover we show that the extremal graphs for $\ftdw{t}54$ and $\ftdw{t}65$ are not independent of $t$. 

We begin with a lemma that gives an upper bound on $\rho_t(G)$ if every vertex $v$ with maximum weight has sufficiently many neighbors that do not have the maximum weight.

\begin{lemma}\label{lem:kk-1}
	Let $G$ be a graph with $\D(G)\le \D$, and suppose $w:V(G)\to\Z$ is a function satisfying (for some $t$) that  
	\[
		k_t(G) = \frac{1}{t}\sum_{v \in V(G)} w(v)
	\]
	If for all $v\in V(G)$ we have  $w(v) \le k$, and also every vertex $v$ with $w(v) = k$ has at least $\ell$ neighbors $u$ with $w(u) \le k-1$, then 
	\[
		\rho_t(G) \le \frac1t\parens[\Big]{k-\frac{\ell}{\ell+\D}}.
	\]
\end{lemma}

\begin{proof}
Partition the vertices of $G$ into $A = \setof{v \in V(G)}{w(v) = k}$ and $B =\setof{v \in V(G)}{w(v) \le k-1}$, with $a = |A|$ so $\abs{B} = n-a$. Consider the cross-edges between $A$ and $B$. Each $v \in A$ has at least $\ell$ edges to $B$. The number $e(A,B)$ of edges between vertices in different classes satisfies
\[
	a\ell \le e(A,B) \le \D(n-a),
\]
so $a\le n\D/(\ell+\D)$. Therefore 
\begin{align*}
		k_t(G)&= \frac1t\sum_{v\in V(G)} w(v) \\
			  &\le \frac1t \parens[\Big]{\frac{\D}{\ell + \D} nk 
			  									+\frac{\ell}{\ell+\D}n(k-1)} \\
			  &= \frac{n}{t}\parens[\Big]{k - \frac{\ell}{\ell+\D}}. \qedhere
\end{align*} 
\end{proof}

We start by defining a new potentially optimal graph for some values of $t, \D$, and $\w$.

\begin{definition}
	Given $k\ge 2$ we define $\BT(k)$ as follows. First let $\BTt(k)$ be obtained from $T_2(2k)$ by deleting an edge and then joining a new vertex to each of the vertices incident with the edge. Then we let $\BT(k) = \BTt(k) \vee E_{k+1}$. Note that $\BTt(k)$ is triangle-free and $\BT(k)\in \cG(2k+1,3)$.
\end{definition}

We first show that for $k= 2, 3$ the graph $\BT(k) \in  \cG(2k+1,3)$ has higher triangle density than the relevant Tur\'an graph, $\T{3k+1}{3}$.

\begin{theorem}\label{thm:2k+1}
	For $k = 2, 3$, 
	\[
		\rho_3(\BT(k)) = \frac{(k+1)(k^2+1)}{3k+2} > \rho_3(\T{3k+1}{3}),
	\]
    and in particular $f_3(2k+1,3) > \rho_3(\T{3k+1}{3})$.
\end{theorem}

\begin{proof}
First note that any triangle in $\BT(k)$ has exactly one edge in $\BTt(k)$ and exactly one vertex in $E_{k+1}$. Therefore 
\[
	(3k+2)\rho_3(\BT(k)) = k_3(\BT(k)) = (k+1)e(\BTt(k)) = (k+1)(k^2 + 1).
\]

For $k = 2, 3$, we have
\[
	\rho_3(\T{3k+1}{3}) = \frac{k^2(k+1)}{3k+1} < \frac{(k+1)(k^2+1)}{3k+2} = \rho_3(\BT(k)).
\]
\end{proof}

\begin{conj}\label{conj:2k+1} We conjecture that for $k = 3$, $\ftdw{3}{2k+1}{3} = \rho_3(\BT(k))$.
\end{conj}

We are able to show that in the case $k=2$ the graph $\BT(2) = C_5\vee E_3$ is optimal.

\begin{theorem}\label{thm:353} 
    $\ftdw{3}{5}{3} = 15/8 = \rho_3(\BT(2))$.
\end{theorem}

\begin{proof}
Let $G \in \cG(5,3)$.
Suppose that $G$ has an optimal vertex $v$, so $N(v) \iso \T52$ and $k_3(v) = 6$. Let $X = \set{x_1, x_2, x_3}$ and $Y = \set{y_1, y_2}$ be the parts of $N(v)$. Then we will show $k_3(x_i) \le 5$ for all $i$.

Suppose to the contrary that $k_3(x) = 6$ for some $x = x_i$, so (by Lemma~\ref{lem:perfectvx}) $x$ is optimal and $N(x) \iso \T{5}{2}$. $N(x)$ contains the independent set $Y = \set{y_1,y_2}$, as well as $v$, which is adjacent to $y_1$ and $y_2$. The vertex $v$ already has the maximum degree, so $Y$ cannot be a subset of the part of size 3 in $N(x)$, and so $Y$ is the part of size 2 in $N(x)$. The vertices $y_i$ then have six neighbors ($x_1, x_2, x_3, v$, and two more vertices in $N(x)$), contradicting $\D(G) \le 5$.

Thus if there is a vertex $v$ with $k_3(v) = 6$, then at least 3 of its neighbors $x$ have $k_3(x) \le 5$. Applying Lemma \ref{lem:kk-1} with $w=k_3$, $\D = 5$, $k=6$, $\ell = 3$, and $t=3$ yields $\rho_3(G) \le \frac{15}{8}$. This upper bound is achieved by $C_5\vee E_3$, which has 8 vertices and 15 triangles, 5 at each vertex of $E_3$. If $G$ has no optimal vertex, then by Lemma \ref{lem:ktvbound}, $\rho_3(G) \le \frac53 < \frac{15}{8}$.
\end{proof}

\subsection{Computing $\ftdw354$ and $\ftdw365$}
\label{sub:computing_ftdw354_and_ftdw365}

In the remainder of this section we establish the values of $\ftdw354$ and $\ftdw365$. The extremal graphs are rather different, and we are uncertain about how these results might generalize.

We start with a lemma concerning graphs in $\cG(5,4)$.

\begin{lemma}\label{lem:7nbrs}In a graph $G \in \cG(5,4)$, every vertex $v$ with $k_3(v) = 7$ has a neighbor $x$ with $k_3(x) \le 5$.\end{lemma}

\begin{proof}
Suppose $k_3(v) = 7$. Then $N(v)$ is a graph with no $K_4$'s, exactly $7$ edges on at most $5$ (and hence exactly $5$ vertices). There are only three such graphs. This can be seen most easily from their complements, which have $5$ vertices, no $E_4$'s, and 3 edges. No vertex has degree $\ge 3$ because then the other four vertices would form an $E_4$. The graphs with maximum degree at most 2 with 3 edges on 5 vertices are $P_3\cup K_2$, $P_4\cup K_1$, and $K_3\cup E_2$.

\begin{case}$N(v) = K_5 - (P_3\cup K_2)$.\end{case}

Let $x$ be the vertex at the center of the deleted $P_3$. Consider $N(x)$. It contains a $P_3$ with $v$ in the center and has at most two more vertices. The two other vertices may be adjacent to each other. Each of the endpoint vertices of the $P_3$ has degree 4 in $N[v]$ so has at most one edge to these two other vertices. Therefore $k_3(x) = e(N(x)) \le 5$.

\begin{case}$N(v) = K_5 - P_4$.\end{case}

Let $u$ be the vertex in $N(v)$ not in the $P_4$, and label the other four vertices $w$, $x$, $y$, $z$ in order along the $P_4$. Consider $N(x)$. It consists of the triangle $\{u,v,z\}$ and up to two vertices outside $N[v]$. The vertex $z$ has four neighbors in $N[v]$ so at most one more edge in $N(x)$, and the vertices $u$ and $v$ already have degree $5=\D(G)$. Therefore $k_3(x) = e(N(x)) \le 5$.

\begin{case}$N(v) = K_5 - K_3$.\end{case}

Within $N(v)$, there are three vertices of degree 2 and two vertices of degree 4. Let $x$ be one of the vertices of degree 2 and $y$ and $z$ be the vertices of degree 4. All of these vertices are also adjacent to $v$, so in $G$, $d(y) = d(z) = d(v) = 5 = \D(G)$, and we cannot add any edges to $v$, $y$, $z$, or $N(v)$. Consider $N(x)$. It contains the triangle $vyz$ and has at most two more vertices, which may be adjacent to each other, but not to $v$, $y$, or $z$. Therefore $k_3(x) = e(N(x)) \le 4$.
\end{proof}

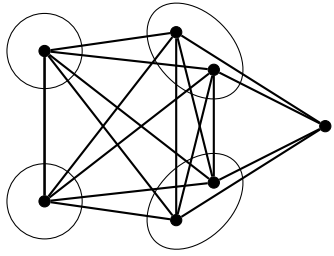
\begin{figure}
    \tikzstyle{vx}=[inner sep=1.5pt,circle,fill=black,draw=black]
    \tikzstyle{edge}=[thick]
    \begin{center}
    \begin{tikzpicture}
    		\draw (1,3) circle [radius=.5];
    		\draw (1,1) circle [radius=.5];
    		\draw[rotate around={45:(3,3)}] (3,3) ellipse [x radius=.5, y radius=.75];
    		\draw[rotate around={-45:(3,1)}] (3,1) ellipse [x radius=.5, y radius=.75];
    		\foreach \x/\y [count=\i] in {2.75/3.25, 3.25/2.75, 3.25/1.25, 2.75/.75, 1/3, 1/1, 4.732/2}
        {
            \node[vx] (V\i) at (\x,\y) {};
        }
    		\foreach \i in {1,2,3,4,5,6}
    		{
    			\foreach \j in {5,6}
    			{
    				\ifthenelse{\i=\j}{}{\draw[edge] (V\i)--(V\j)};
    			}
    		}
    		\draw[edge] (V1)--(V3);
    		\draw[edge] (V1)--(V4);
    		\draw[edge] (V2)--(V3);
    		\draw[edge] (V2)--(V4);
    		\foreach \i in {1,2,3,4}{
    			\draw[edge] (V\i)--(V7);
    		}
    \end{tikzpicture}
    \end{center}
    \caption{The graph $G^*$.}
    \label{fig:354}
\end{figure}

The optimal graph for achieving $\ftdw354$ is constructed by taking $K_6$ minus a matching of size $2$, and then joining the vertices incident to matching edges to a new vertex. This graph $G^*$ is illustrated in Figure \ref{fig:354}.

\begin{theorem}\label{thm:354}
    $\ftdw{3}{5}{4} = 16/7 = \rho_3(G^*)$.
\end{theorem}

\begin{proof}
First note that $\rho_3(G) \le \rho_3(\T64)$ for all graphs $G$ on at most $6$ vertices by Lemma \ref{lem:basecase}, and $\rho_3(\T64) < \rho_3(G^*)$. Suppose $G \in \cG(5,4)$ and set $n = \abs{V(G)}$. We may assume that $G$ attains the maximum number of triangles among graphs with $n$ vertices in $\cG(5,4)$, and further that every edge of $G$ is in a triangle: any edge of $G$ not in a triangle may be deleted without changing the number of triangles, the number of vertices, $\D(G) \le 5$, and $\w(G) \le 4$.

Suppose $G$ has an optimal vertex $v$, so $N(v) \iso \T53$ and $k_3(v) = 8$. Consider the possible continuations of $N[v] \iso\T64$ in $G$. The two vertices in parts of size one already have the maximum degree, 5, and 3-weight 8. The four vertices in parts of size two $\set{w, x}$ and $\set{y, z}$ have four neighbors in the $\T64$ so may have one more neighbor each. If there is a vertex $u \in G\setminus N[v]$ adjacent to any of these four border vertices, without loss of generality say $w$, then $u$ must also be adjacent to $y$ or $z$ so that $uw$ is in a triangle.

If there is at most one triangle containing vertices in both $N[v]$ and $G\setminus N[v]$, then $k_3(G) \le 12 + 1 + \frac{16}7(n-6) < \frac{16}7n$ by induction on $n$. Otherwise, there are at least two triangles containing vertices in both $N[v]$ and $G\setminus N[v]$. If $w$, $x$, $y$, and $z$ have only one neighbor $u$ outside $N[v]$, then $u$ must be adjacent to three or four of them. In either case, let $H = G[N[v]\cup\set{u}]$. If $u$ has three neighbors in $N(v)$, then $k_3(H) = 14$ and $u$ has at most two neighbors outside $N[v]$, so $k_3(G) \le 14 + 1 + \frac{16}{7}(n-7) < \frac{16}7n$ by induction. If $u$ has four neighbors in $N(v)$, then $H = G^*$ and $k_3(H) = 16$. The only border vertex, $u$, has at most one neighbor outside $N[v]$, so $H$ is detachable by Lemma \ref{lem:detach}. We have $\rho_3(H) = \frac{16}7$, and by induction $\rho_3(G\setminus H) \le \frac{16}7$, so by Observation \ref{obs:induction}, $\rho(G)\le \frac{16}7$.

The vertices $w$, $x$, $y$, and $z$ cannot have more than two neighbors outside $N[v]$ since every edge is in a triangle. We have shown that $\rho_3(G) \le \frac{16}7$ for every $G$ except when for every optimal vertex $v$, the vertices $w$, $x$, $y$, and $z$ have two neighbors outside $N[v]$, $t$ and $u$. Without loss of generality, $G$ has the edges $tw, ty, ux,$ and $uz$. The vertices $t$ and $u$ may or may not be adjacent. We call $C \subseteq V(G)$ an \emph{$8$-configuration} if $C$ has size 8 and is isomorphic to $N[v]\cup\set{t,u}$ described above, where $t$ and $u$ may or may not be adjacent (see Figure~\ref{fig:8conf}).

\begin{figure}
    \tikzstyle{vx}=[inner sep=1.5pt,circle,fill=black,draw=black]
    \tikzstyle{edge}=[thick]
    \begin{center}
    \begin{tikzpicture}
		\draw (1,3) circle [radius=.5];
		\draw (1,1) circle [radius=.5];
		\draw[rotate around={45:(3,3)}] (3,3) ellipse [x radius=.5, y radius=.75];
		\draw[rotate around={-45:(3,1)}] (3,1) ellipse [x radius=.5, y radius=.75];
		\foreach \x/\y [count=\i] in {2.75/3.25, 3.25/2.75, 3.25/1.25, 2.75/.75, 1/3, 1/1}
        {
            \node[vx] (V\i) at (\x,\y) {};
        }
		\foreach \i/\lab/\x/\y in {7/u/4.5/0.75, 8/t/4.5/3.25}
        {
            \node[vx, label=right:$\lab$] (V\i) at (\x,\y) {};
        }
		\foreach \i in {1,2,3,4,5,6}
		{
			\foreach \j in {5,6}
			{
				\ifthenelse{\i=\j}{}{\draw[edge] (V\i)--(V\j)};
			}
		}
		\draw[edge] (V1)--(V3)
		            (V1)--(V4)
		            (V2)--(V3)
		            (V2)--(V4)
                    (V2)--(V7)
                    (V1)--(V8)
                    (V4)--(V7)
                    (V3)--(V8);
        \draw[edge, dashed] (V7)--(V8);
    \end{tikzpicture}
    \end{center}
    \caption{An $8$-configuration. The dashed edge may or may not be present.}
    \label{fig:8conf}
\end{figure}

In an $8$-configuration $C$, $t$ and $u$ are the border vertices. They each have at least two neighbors in $C$, so at most three neighbors outside $C$, and no more than three cross-triangles. If $H$ shares a border vertex with another $8$-configuration, then that border vertex has only one cross-triangle. Then $k_3(G[C]) = 14$ and $k_3(G) \le 14 + 3 + 1 + \frac{16}7(n-8) < \frac{16}7n$ by induction. Therefore we may assume every border vertex of an $8$-configuration is in only one $8$-configuration.

Assign a weight to each vertex $v$ in $G$ by \[w(v)=\begin{cases}k_3(v)+2&\text{if $v$ is a border vertex of an $8$-configuration}\\6& \text{if $k_3(v) = 8$, i.e. $v$ is optimal}\\k_3(v)&\text{otherwise.}\end{cases}\] We have assumed that every optimal vertex $v$ is in an $8$-configuration. Each $8$-configuration has two optimal vertices and two border vertices, so \[k_3(G) = \frac13\sum_{v \in V(G)}k_3(v) = \frac13\sum_{v \in V(G)}w(v).\]

A border vertex $u$ of an $8$-configuration has $k_3(u) \le 4$, so $w(u) \le 6$. Every vertex $v$ with $w(v) = 7$ has $k_3(v) = 7$ so by Lemma \ref{lem:7nbrs} has a neighbor $x$ with $k_3(x) \le 6$ and $w(x) \le 6$. Using the function $w$, $t=3$, $\D = 5$, $\ell = 1$, and $k = 7$, Lemma \ref{lem:kk-1} gives $\rho_3(G) \le \frac{41}{18} <\frac{16}7$.
\end{proof}

Finally we determine $\ftdw365$, which is in fact attained by a graph having no $K_5$'s at all! 

\begin{theorem}\label{thm:365}
    $\ftdw365 = 4 = \rho_3(\T84)$.
\end{theorem}

\begin{proof}
First note that $\rho_3(G) \le \rho_3(\T75)$ for all graphs $G$ on at most $7$ vertices by Lemma \ref{lem:basecase}, and $\rho_3(\T75) < \rho_3(\T84)$. 
Suppose $G$ has an optimal vertex $v$, so $N(v) \iso \T64$ and $k_3(v) = 13$. Then $N[v] \iso\T75$. The vertices in the parts of size 1 all are optimal. The border vertices are the vertices in the parts of size 2. Any border vertex $x$ has in its neighborhood a $\T54$ from $N[v]$ and at most one more vertex, which cannot be adjacent to the optimal vertices in $N[v]$. Therefore $k_3(x) = e(N(x)) \le 11$.

We call $C \subseteq V(G)$ a \emph{$7$-configuration} if $G[C] \iso \T75$. Assign a weight to each vertex $v$ in $G$ by \[w(v)=\begin{cases}12&\text{if $k_3(v) = 13$, i.e. $v$ is optimal}\\k_3(v)+\frac34& \text{if $v$ is a border vertex of a $7$-configuration}\\k_3(v)&\text{otherwise.}\end{cases}\]

A border vertex $x$ of a $7$-configuration can have at most one cross-edge and cannot have an optimal neighbor $y$ outside the $7$-configuration. If it did, then it would also have to share some of $y$'s neighbors, since in $N[y] \iso \T75$ every vertex has degree at least five, contradicting $d(x) \le 6$. Therefore the $7$-configurations are disjoint, with three optimal vertices for every four border vertices of $7$-configurations, and \[k_3(G) = \frac13\sum_{v \in V(G)}k_3(v) = \frac13\sum_{v \in V(G)}w(v).\]

Observe that $w(v) \le 12$ for all $v$ since the border vertices of $7$-configurations $x$ have $k_3(x) \le 11$, so $\frac13\sum_{v \in V(G)}w(v) \le 4n$ and $\rho_3(G) \le 4$. If $G$ contains an optimal vertex, then in addition there must be some vertex $v$ (in its $7$-configuration) with $w(v) < 12$, and $\rho_3(G) < 4$. This upper bound is achieved by $\T84$, which has $k_3(v) = 12$ for every vertex $v$.
\end{proof}

\section{Open Problems}\label{sec:open}

The results of this paper can be summarized in the table below, where each entry indicates the section in which the extremal graph is determined. For results from Sections~\ref{sec:bounds} and \ref{sec:d=w} the extremal graphs are the same for each $3 \le t \le \w$. From Section~\ref{sec:max} our results determine only $\ftdw{r-1}{r+1}{r}$ and $\ftdw{r}{r+1}{r}$.
\medskip
\begin{center}
\begin{tabular}{rccccccccc}\toprule
     & \multicolumn{8}{c}{$\D$} \\ \cmidrule(r){2-9} 
$\w$ & 3 & 4 & 5 & 6 & 7 & 8 & 9 & 10\\ \midrule
3 & \ref{sec:d=w} & \ref{sec:bounds} & \ref{sec:other} & \ref{sec:bounds} &  & \ref{sec:bounds} &  & \ref{sec:bounds}\\
4 & \ref{sec:bounds} & \ref{sec:d=w} & \ref{sec:other} & \ref{sec:bounds} &  &  & \ref{sec:bounds} & \\
5 &                  & \ref{sec:bounds} & \ref{sec:d=w} & \ref{sec:other} &  & \ref{sec:bounds} &  &\\
6 &                   &                   & \ref{sec:bounds} & \ref{sec:d=w} & \ref{sec:max} &  &  &  \ref{sec:bounds} \\
7 &                   &                   &                   & \ref{sec:bounds} & \ref{sec:d=w} & \ref{sec:max} & &   \\
8 &                  &                   &                  &                  & \ref{sec:bounds} & \ref{sec:d=w} & \ref{sec:max} &  \\
9 &                  &                  &                  &                  &                  & \ref{sec:bounds} & \ref{sec:d=w} & \ref{sec:max} \\
10 &                  &                  &                  &                  &                  &                  & \ref{sec:bounds} & \ref{sec:d=w} \\ \bottomrule
\end{tabular}
\end{center}
\medskip
Of course, we would like to complete the entire table for all $t$. Given the $(\w-1)$-periodicity of the results of Section \ref{sec:bounds}, it might be possible to generalize the results of Section \ref{sec:d=w} to find $\ftdw{t}{a(r-1)+1}{r}$ for integers $a > 1$, or the results of Section \ref{sec:max} to find $\ftdw{t}{a(r-1)+2}{r}$ for integers $a > 1$ and $t = r-1$ or $r$. The upper and lower bounds appear to be closer for smaller values of $\D \mod \w-1$.

In Sections \ref{sec:bounds}, \ref{sec:d=w}, and \ref{sec:max}, we found that the extremal graphs were all the lower bound graph $L(\D,\w)$ from Section~\ref{sec:bounds}, of the form $\T{\D+a}{\w}$, where $a = \floor{\frac{\D}{\w-1}}$. However, we also identified triples $(t,\D,\w)$ where $\T{\D+a}{\w}$ is not extremal. For $(3,5,3)$, $(3,5,4)$, and $(3,6,5)$, we determined the extremal graphs. Theorem~\ref{thm:2k+1} gives a graph in $\cG(7,3)$ that has more triangles per vertex than $L(7,3)$. Is this graph extremal? What conditions on $(t,\D,\w)$ would imply that $\T{\D+a}{\w}$ is or is not extremal?

When restricting $\w(G)$ but not $\D(G)$, Zykov's Theorem shows that the same graphs maximize the number of cliques of every size. Similarly, when restricting $\D(G)$ but not $\w(G)$, Gan, Loh, and Sudakov conjecture that the same graphs maximize the number of cliques of every size $t \ge 3$, and this conjecture has been proven in significant cases. A corollary of the Kruskal-Katona Theorem shows that among graphs on $m$ edges, the colex graph $\cC(m)$ maximizes the number of cliques of every size. It is surprising then that for some pairs $(\D,\w)$, namely $(5,4)$ and $(6,5)$, our extremal graphs are not the same for every $t$. It would be nice to determine when one graph works for all $t$.

Finally, we have not established even that for all relevant values of $t, \D$, and $\w$ there is a graph $G$ achieving $\rho_t(G)=\ftdwdef$. We would like to prove the following conjecture.
\begin{conj}
    For all $t\ge 3$ and $3\le \w \le \D+1$ there exists a graph $G\in \cG(\D,\w)$ such that $\rho_t(G) = \ftdwdef$.
\end{conj}

\bibliographystyle{plain}
\bibliography{Bibliography}
\end{document}